\documentclass[11pt]{article}
\usepackage{preamble}

\title{Infinitary positive existential normal forms for modules}
\author{Rishi Banerjee%
\thanks{This material is based upon work supported by the National Science Foundation under Award No. DMS-2452105.}}
\date{August 27, 2026}

\begin{document}
\maketitle

\begin{abstract}
    Let $\theta$ be a regular cardinal, $R$ a ring, and $M$ a left $R$-module. We prove that for every ordinal $\alpha$ there is a set $I_\alpha \subseteq M^{<\theta}$ of size at most $\beth_\alpha(|R| + \theta)$ such that every parameter-free $L_{\infty,\theta}$ formula of rank at most $\alpha$ is equivalent in $M$ to an infinitary Boolean combination of cosets $\abar + \phi(M)$, where $\abar \in I_\alpha$ and $\phi$ is an infinitary positive existential formula of rank at most $\alpha$. 
    The main ingredient in the proof is a combinatorial lemma which says that given at most $\kappa$ subgroups of an abelian group, there is a set of at most $2^\kappa$ points which tests whether any family in which each member is either empty or a coset of the corresponding subgroup covers the whole group.
    The proof proceeds by applying this lemma fiberwise to show that the relevant complete Boolean algebras of positive-existentially definable cosets are closed under projections. 
\end{abstract}

\section*{Acknowledgments}
    I would like to thank John Baldwin for introducing me to this problem and for extensive feedback on the paper. I am also grateful to Matthew Harrison-Trainor and Andr\'es Villaveces for their feedback.

\section{Introduction}

Positive-primitive (pp) formulas play a central role in the first-order model theory of abelian groups and modules. A first-order formula is pp if it has the form 
\[ \exists \ybar \bigwedge_{i < n} \phi_i(\xbar, \ybar), \]
where the $\phi_i$ are atomic formulas. 
In a module $M$, parameter-free pp formulas define subgroups of finite powers of $M$, and consistent pp formulas with parameters define cosets of such subgroups.

Szmielew proved that every first-order formula in the language of abelian groups is equivalent modulo the theory of abelian groups to a finite Boolean combination of pp formulas and sentences recording natural invariants such as the group exponent and the dimensions of certain quotients \cite{Szmielew1955}. In a fixed group $G$, the invariant sentences have fixed truth values, so every formula is equivalent in $G$ to a finite Boolean combination of pp formulas.\footnote{Szmielew does not use the term ``positive primitive'', and instead refers to equality and congruence functions, but these are pp-definable.}

Baur generalized this result to modules over arbitrary rings, while Fisher developed the corresponding theory for the still more general setting of abelian structures \cite{Baur1976,Fisher1977}. Given a ring $R$, every first-order formula about left $R$-modules is equivalent over the theory of left $R$-modules to a finite Boolean combination of pp formulas and $\forall \exists$ sentences. Thus in a fixed module $M$, every first-order formula is equivalent to a finite Boolean combination of pp formulas.

In \cite{Shelah2012}, Shelah considered an analogous theorem for the infinitary logic $L_{\infty,\theta}$, where $\theta$ is a regular cardinal. In place of pp formulas, he defines a transfinite hierarchy of infinitary positive existential (pe) formulas, obtained by starting with homogeneous linear equations and repeatedly taking set-sized conjunctions and applying blocks of existential quantifiers of length $< \theta$. The precise definition of the pe hierarchy is given in the next section. The proposed theorem in \cite{Shelah2012}, which is also the main theorem in the present paper, is the following. 

\begin{theorem}[Main theorem]
\label{thm:inf-qe}
    Let $R$ be a ring, $\theta$ a regular cardinal, and $M$ a left $R$-module. 
    There is an increasing sequence of sets of parameter tuples
    \[
    (I_\alpha)_{\alpha \in \Ord}, \qquad I_\alpha \subseteq M^{<\theta}, \qquad |I_\alpha| \leq \beth_\alpha(|R| + \theta),
    \]
   such that every parameter-free $L_{\infty,\theta}$ formula $\phi(\xbar)$ of rank at most $\alpha$  is equivalent in $M$ to an infinitary Boolean combination of formulas of the form $\psi(\xbar - \abar)$, where $\abar \in I_\alpha$ and $\psi(\xbar)$ is pe of rank at most $\alpha$. 
\end{theorem}

Thus in a fixed module $M$, for any $\alpha$, we have quantifier elimination for $L_{\infty,\theta}$ formulas of rank at most $\alpha$ after expanding the language by adding at most $\beth_\alpha(|R| + \theta)$ new relation and constant symbols naming the subgroups defined by formulas in $\Lambda_\alpha$ and the parameters we translate them by. The need for model-dependent parameters even when the original formula is parameter-free is an important difference from the first-order case, and the content of the theorem is that the number of parameters needed can be bounded independently of $M$.

A gap in Shelah's 2012 proof of \cref{thm:inf-qe} was identified by John Baldwin in 2026. A revised version of Shelah’s paper is available in his online archive under the identifier Sh:977. Here we give an independent proof by a different method, based on a simple combinatorial lemma which provides a small set of ``test points'' for deciding whether a family of cosets covers an abelian group.

\section{Infinitary positive existential normal forms}

\subsection{Infinitary positive existential formulas}

Fix a regular cardinal $\theta$ and a ring $R$. Let $\tau_R$ be the language of left $R$-modules, and fix a set of variables $\set{x_i \mid i \in \theta}$. For any ordinal $\alpha$, let 
\[ \kappa_\alpha = \beth_\alpha(|R| + \theta). \]

We work in the infinitary logic $L_{\infty,\theta}$ where arbitrary set-sized conjunctions and disjunctions are allowed, but existential and universal quantifiers can only occur in blocks of length $< \theta$. We write $L_{\infty,\theta,\alpha}$ for the set of $L_{\infty,\theta}$ formulas of quantifier rank at most $\alpha$. The notion of quantifier rank we use here counts only quantifiers as increasing the rank, not conjunctions or disjunctions.
\begin{itemize}
    \item Atomic formulas are rank $0$.

    \item Negations preserve rank.

    \item The ranks of $\bigvee_{\phi \in \Phi} \phi$ and $\bigwedge_{\phi \in \Phi} \phi$ are both $\sup_{\phi \in \Phi} \rank(\phi)$.

    \item If $\phi(\xbar,\ybar)$ has rank $\alpha$ and does not start with an existential quantifier, then the rank of $\exists \ybar \phi(\xbar,\ybar)$ is $\alpha + 1$.

    \item If $\phi(\xbar,\ybar)$ has rank $\alpha$ and does not start with a universal quantifier, then the rank of $\forall \ybar \phi(\xbar,\ybar)$ is $\alpha + 1$.
\end{itemize}

Following Shelah, we define a hierarchy of infinitary positive existential (pe) formulas in the logic $L_{\infty,\theta}$ as follows. For any ordinal $\alpha$ and $\epsilon < \theta$, write $\Lambda_{\alpha,\epsilon}$ for the set of pe formulas of rank at most $\alpha$ with free variables in $\set{x_i \mid i \in \epsilon}$.
\begin{itemize}
    \item $\Lambda_{0,\epsilon}$ is the set of homogeneous $R$-linear equations in the first $\epsilon$ free variables:
        \[ a_1 x_{\xi_1} + \dots + a_n x_{\xi_n} = 0, \]
    where $a_1,\dots,a_n \in R$ and $\xi_1,\dots,\xi_n < \epsilon$.
    
    \item $\Lambda_{\alpha+1,\epsilon}$ is the set of formulas of the form 
    \[ 
        \exists \ybar \bigwedge_{\phi \in \Phi} \phi(\xbar,\ybar), 
    \]
    where $|\ybar| = \zeta < \theta$ and $\Phi \subseteq \Lambda_{\alpha, \epsilon+\zeta}$. (Note that $\Lambda_{\alpha,\epsilon} \subseteq \Lambda_{\alpha+1,\epsilon}$ by taking $\zeta = 0$ and letting $\Phi$ be a singleton.)
    
    \item For a limit ordinal $\lambda$, 
    \[ \Lambda_{\lambda,\epsilon} = \bigcup_{\beta < \lambda} \Lambda_{\beta,\epsilon}. \]
\end{itemize}

Also, put $\Lambda_\alpha = \bigcup_{\epsilon < \theta} \Lambda_{\alpha,\epsilon}$.

\begin{proposition}
    $|\Lambda_\alpha| \leq \kappa_\alpha$.
\end{proposition}

\begin{proof}
    At rank $0$ there are 
    \[ 
        |R \times \theta|^{<\omega} = |R| + |\theta| = \kappa_0
    \]
    homogeneous linear equations in the variables $\set{x_i}_{i < \theta}$.

    At a limit ordinal $\lambda$, 
    \[ 
    |\Lambda_\lambda| 
    \leq \sum_{\beta < \lambda} |\Lambda_{\beta}| 
    \leq |\lambda| \cdot \sup_{\beta < \lambda} \kappa_\beta 
    = \kappa_\lambda. 
    \]

    For the successor case, a formula in $\Lambda_{\alpha+1,\epsilon}$ is determined by a choice of $\zeta < \theta$ together with a subset $\Phi \subseteq \Lambda_{\alpha,\epsilon+\zeta}$, so
    \[
        |\Lambda_{\alpha+1}| 
        \leq \sum_{\epsilon,\zeta < \theta} 2^{|\Lambda_{\alpha,\epsilon+\zeta}|}
        \leq \theta^2 \cdot 2^{\kappa_\alpha} 
        = \kappa_{\alpha+1}. \qedhere
    \]
\end{proof}

In a left $R$-module $M$, any formula $\phi(\xbar) \in \Lambda_{\alpha,\epsilon}$ defines a subgroup $\phi(M) \leq M^\epsilon$. This is because the set of solutions of a homogeneous linear equation is a subgroup, conjunctions and existential quantifiers correspond semantically to intersections and projections, and intersections and projections of subgroups are still subgroups. 

Translates of pe formulas define cosets of the corresponding subgroups. If $\abar \in M^\epsilon$ then $\phi(\xbar - \abar)$ defines the coset $\abar + \phi(M)$. Given a set of tuples $I \subseteq M^{<\theta}$, let
\begin{align*}
    \Cos_{\alpha,\epsilon}(I) &= \set{ \abar + \phi(M) \mid \abar \in I \cap M^\epsilon, \phi \in \Lambda_{\alpha,\epsilon} }, \\
    \Cos_\alpha(I) &= \bigcup_{\epsilon < \theta} \Cos_{\alpha,\epsilon}(I).
\end{align*}

We will be considering Boolean combinations of such cosets. Given a set $X$ and $\mc{C} \subseteq \mc{P}(X)$, let $\Bool_\infty(\mc{C})$ be the least complete Boolean subalgebra of $\mc{P}(X)$ containing $\mc{C}$. 
Every element of $\Bool_\infty(\mc{C})$ can be written as a union of Boolean atoms, where a Boolean atom is a nonempty set of the form 
\[ 
    Q_\eta = \bigcap_{\eta(C) = 1} C \setminus \bigcup_{\eta(C) = 0} C
\]
for $\eta: \mc{C} \to 2$. 

With this notation, the conclusion of \cref{thm:inf-qe} can be written
\[
    \phi(M) \in \Bool_\infty(\Cos_{\alpha,\epsilon}(I_\alpha)).
\]

\begin{remark}[Small models and stability consequences]
    If $|M^{<\theta}| \leq \kappa_\alpha$ and $\alpha \geq 1$, then the conclusion of the theorem follows trivially by setting $I_\alpha = M^{<\theta}$. Then every singleton is in $\Cos_\alpha(I_\alpha)$. And for any $\epsilon < \theta$, every subset of $M^\epsilon$ (definable or not) is a union of singletons, so belongs to $\Bool_\infty( \Cos_{\alpha,\epsilon}(I_\alpha))$. So the substantive case of the theorem is $|M^{<\theta}| > \kappa_\alpha$.

    This does not conflict with the stability results in \cite{Shelah2012}. In the expansion of the small model, every relation is definable, but there are not enough tuples to realize a linear order of length $(\kappa_{\alpha+1})^+$.
\end{remark}

The proof of \cref{thm:inf-qe} is by induction, and the only non-trivial case is the quantifier case.
Proving this case amounts to showing that the projection of a Boolean combination of cosets in $\Cos_{\alpha,\epsilon+\zeta}(I_\alpha)$, can be written as a Boolean combination of cosets in $\Cos_{\alpha+1,\epsilon}(I_{\alpha+1})$.
Using the fact that any element of a complete Boolean algebra can be written in disjunctive normal form, this simplifies to a special case which we prove in \cref{lem:proj-atom} where the set being projected is a single Boolean atom in $\Bool_\infty(\Cos_{\alpha,\epsilon+\zeta}(I_\alpha))$, i.e. a set of the form
\[ \bigcap_{i \in I} C_i \setminus \bigcup_{j \in J} D_j, \]
where $\set{C_i \mid i \in I}$ and $\set{D_j \mid j \in J}$ form a partition of $\Cos_{\alpha,\epsilon+\zeta}(I_\alpha)$.
Finally, the proof of \cref{lem:proj-atom} uses \cref{lem:test-set}, which shows the existence of a test set of size $2^\kappa$ for determining if a group is covered by cosets of a fixed collection of $\kappa$ many subgroups.
This will give us a pe way of expressing the condition that $\bigcup_{j \in J} D_j$ covers $\bigcap_{i \in I} C_i$ using only $2^{\kappa_\alpha}$ many new parameters. 

We begin with \cref{lem:test-set}.

\subsection{A small test set for covers by cosets}

\begin{definition}
    Given a group $G$ and a family of subgroups $(G_s)_{s \in S}$, say that a set $T \subseteq G$ is a \textbf{test set for $(G,(G_s)_{s \in S})$} if, for every family $(C_s)_{s \in S}$ where each $C_s$ is either empty or a coset of $G_s$,
    \[ 
        G = \bigcup_{s \in S} C_s \iff T \subseteq \bigcup_{s \in S} C_s.
    \]
\end{definition}

\begin{lemma}
\label{lem:test-set}
    Let $\kappa$ be an infinite cardinal, let $G$ be an abelian group, and let $(G_s)_{s \in S}$ be a family of subgroups of $G$ with $|S| \leq \kappa$. Then there is a test set $T$ for $(G, (G_s)_{s \in S})$ such that $|T| \leq 2^\kappa$.
\end{lemma}

\begin{proof}
    Suppose that there is no such test set $T$. 
    Then we can define by transfinite induction for each $\alpha < (2^\kappa)^+$ an element $g_\alpha \in G$ and a family $(C_s^\alpha)_{s \in S}$ such that each $C_s^\alpha$ is either empty or a coset of $G_s$, 
    \[ g_\beta \in \bigcup_{s \in S} C_s^\alpha \text{ for all } \beta < \alpha, \] 
    and 
    \[ g_\alpha \not \in \bigcup_{s \in S} C_s^\alpha. \tag{$\ast$} \label{eq:test-set-contradiction} \]
    This is possible because at stage $\alpha$, $\set{g_\beta : \beta < \alpha}$ has cardinality at most $2^\kappa$, so it is not a test set by assumption. 
    This is witnessed by a family $(C_s^\alpha)_{s \in S}$ of the appropriate form which covers $\set{g_\beta : \beta < \alpha}$ but does not cover $G$; let $g_\alpha$ be any element of $G \setminus \bigcup_{s \in S} C_s^\alpha$.

    Now for each $\beta < \alpha < (2^\kappa)^+$, pick a color $c(\beta,\alpha) \in S$ such that $g_\beta \in C_{c(\beta,\alpha)}^\alpha$. 
    By the Erd\H{o}s-Rado theorem \cite{ER1956},
    \[ 
        (2^\kappa)^+ \to (3)^2_\kappa, 
    \]
    there is a monochromatic triangle, $\beta_0 < \beta_1 < \beta_2 < (2^\kappa)^+$ such that all pairs $(\beta_0,\beta_1), (\beta_0,\beta_2), (\beta_1,\beta_2)$ have the same color $s \in S$. 
    Then $g_{\beta_0}, g_{\beta_1}$ are both in the coset $C_s^{\beta_2}$, so $g_{\beta_1} - g_{\beta_0} \in G_s$. 
    But $g_{\beta_0} \in C_s^{\beta_1}$, so $g_{\beta_1} \in C_s^{\beta_1}$, contradicting \eqref{eq:test-set-contradiction}.
\end{proof}

\begin{remark}
    The same result holds for any family $(E_s)_{s \in S}$ of at most $\kappa$ equivalence relations on a set $X$, replacing ``coset of $G_s$'' with ``$E_s$-class''.
\end{remark}

\begin{example}
    The following example shows that the bound on the test set $T$ in \cref{lem:test-set} is sharp. 

    Let $G = 2^\kappa$ with the group operation being coordinatewise addition modulo 2. For each $\alpha < \kappa$, let $G_\alpha = \set{g \in G \mid g(\alpha) = 0} \leq G$. Suppose $T \subseteq G$ with $|T| < 2^\kappa$. Then there is some $x \in G \setminus T$. For each $\alpha < \kappa$, let $i_\alpha \in G$ be the indicator function of the singleton $\set{\alpha}$, and let
    \[
    C_\alpha = \set{g \in G \mid g(\alpha) \neq x(\alpha)} = (x + i_\alpha) + G_\alpha. 
    \]
    Then each $C_\alpha$ is a coset of $G_\alpha$, but $\bigcup_{\alpha < \kappa} C_\alpha = G \setminus \set{x}$, so the $C_\alpha$'s cover $T$ but not $G$.
\end{example}

\subsection{Projection of one Boolean atom}

Let $M$ be a module, $\epsilon,\zeta < \theta$, and let $\pi: M^{\epsilon + \zeta} \to M^\epsilon$ be the projection map which forgets the last $\zeta$ coordinates.

\begin{lemma}
\label{lem:proj-atom}
    Let $\kappa$ be infinite and let 
    \[ Q = \bigcap_{i \in I} C_i \setminus \bigcup_{j \in J} D_j, \]
    where the $C_i$ and $D_j$ are cosets of subgroups $P_i, N_j \leq M^{\epsilon+ \zeta}$ and $|I|, |J| \leq \kappa$.
    
    If $Q \neq \emptyset$, then there is $a \in M^{\epsilon + \zeta}$, an index set $K$ with $|K| \leq 2^\kappa$, a family $(b_k)_{k \in K}$ where each $b_k \in M^{\epsilon + \zeta}$, and a family $(J_k)_{k \in K}$ where each $J_k \subseteq J$, such that 
    \[ 
        \pi(Q) = \left( \pi(a) + \pi(P) \right) \setminus \bigcup_{k \in K}  \left[ \pi(b_k) + \pi \left( P \cap \bigcap_{j \in J_k} N_j \right) \right],
    \]
    where $P = \bigcap_i P_i$.
\end{lemma}

\begin{proof}
    Let $C = \bigcap_{i \in I} C_i$, and note that $C$ is a coset of $P$. 
    Pick any point $a \in C$. 
    Then $\pi(C) = \pi(a) + \pi(P)$.
    
    Let $x \in \pi(C)$, and let $C_x, (D_j)_x \subseteq M^\zeta$ be the fibers over $x$:
    \begin{align*}
        C_x  = \set{ y \in M^\zeta \mid (x,y) \in C }, \qquad
        (D_j)_x  = \set{ y \in M^\zeta \mid (x,y) \in D_j }.
    \end{align*}
    Then 
    \[
        x \not \in \pi(Q) \iff C_x \subseteq \bigcup_{j \in J} (D_j)_x.
    \]
    Let $G, G_j$ be the vertical subgroups of $P, N_j$:
    \begin{align*}
        G = \set{ g \in M^\zeta \mid (0,g) \in P }, \qquad
        G_j = \set{ g \in M^\zeta \mid (0,g) \in N_j }.
    \end{align*}
    For any $y \in C_x$, $C_x - y = G$ and $(D_j)_x - y$ is empty or a coset of $G_j$.
    Let 
    \[
        B_j(x,y) = \set{ g \in G \mid y + g \in (D_j)_x }.
    \]
    For each $j$, $B_j(x,y)$ is either empty or a coset of $G \cap G_j$, and
    \[
        C_x \subseteq \bigcup_{j \in J} (D_j)_x \iff G \subseteq \bigcup_{j \in J} B_j(x,y).
    \]
    Applying \cref{lem:test-set} with $S = J$ to $(G, (G \cap G_j)_{j \in J})$, we obtain a test set $T \subseteq G$ with $|T| \leq 2^\kappa$ and 
    \[ 
        G \subseteq \bigcup_{j \in J} B_j(x,y) \iff T \subseteq \bigcup_{j \in J} B_j(x,y). \]
    Note that 
    \[ 
        t \in B_j(x,y) \iff (x,y) \in D_j - (0,t).
    \]
    So we have
    \begin{align*}
        x \not \in \pi(Q) 
        & \iff \exists y \in C_x \ \left[ T \subseteq \bigcup_{j \in J} B_j(x,y) \right] \\
        & \iff \exists y \in C_x \ \exists f: T \to J \ \forall t \in T \ \left[ (x,y) \in D_{f(t)} - (0,t) \right] \\
        & \iff \exists y \in C_x \ \exists f: T \to J \ \left[ (x,y) \in \bigcap_{t \in T} (D_{f(t)} - (0,t)) \right] \\
        & \iff x \in \bigcup_{f: T \to J} \pi \! \left( C \cap \bigcap_{t \in T} (D_{f(t)} - (0,t)) \right).
    \end{align*}
    For any $f: T \to J$, let
    \[ E_f = C \cap \bigcap_{t \in T} (D_{f(t)} - (0,t)), \]
    so the above becomes
    \[
        x \not \in \pi(Q) \iff x \in \bigcup_{f: T \to J} \pi(E_f).
    \]
    Each $E_f$ is either empty or a coset of $P \cap \bigcap_{t \in T} N_{f(t)} = P \cap \bigcap_{j \in f(T)} N_j$.
    For each nonempty $E_f$, pick an element $b_f \in E_f$, so that 
    \[  
        E_f = b_f + \left( P \cap \bigcap_{j \in f(T)} N_j \right).
    \]
    This gives us the following normal form for $\pi(Q)$:
    \begin{align*}
        \pi(Q) = \pi(C) \setminus \bigcup_{\substack{f: T \to J \\ E_f \neq \emptyset}} \left[ \pi(b_f) + \pi\! \left( P \cap \bigcap_{j \in f(T)} N_j \right) \right].
    \end{align*}

    There is one final bit of cardinality book-keeping to attend to. 
    We promised at most $2^\kappa$ many points $b_f$ and subsets $J_f$, but there may be more than $2^\kappa$ functions $T \to J$. However, we claim that there are only $2^\kappa$ many distinct nonempty $E_f$'s. 
    To see this, first rewrite $E_f$ as
    \[ 
        E_f = C \cap \bigcap_{t \in T} (D_{f(t)} - (0,t)) = C \cap \bigcap_{j \in f(T)} \bigcap_{t \in f^{-1}(j)} (D_j - (0,t)).
    \]
    Now if $f(t) = f(t') = j$, then $D_j - (0,t)$ and $D_j - (0,t')$ are both cosets of $N_j$.
    And if $E_f \neq \emptyset$, then these two cosets intersect, so they must be equal.
    So it suffices to pick for each $j \in f(T)$ a single $t_j \in f^{-1}(j) \subseteq T$. Define $\tau: J \to T \cup \set{\uparrow}$ by 
    \[
        \tau(j) = \begin{cases}
            t_j, & j \in f(T) \\
            \uparrow, & j \not \in f(T)
        \end{cases}.
    \]
    Then
    \[ 
        E_f = C \cap \bigcap_{j \in \tau^{-1}(T)} (D_j - (0,\tau(j))). 
    \]
    Thus every nonempty $E_f$ is determined by some map $\tau: J \to T \sqcup \set{\uparrow}$, and the number of such maps is at most 
    \[ (|T| + 1)^{|J|} \leq (2^\kappa)^\kappa = 2^\kappa. \]
    
    So, we may index the distinct nonempty sets $E_f$ as $\set{E_k \mid k \in K}$, where $|K| \leq 2^\kappa$. For each $k \in K$ choose a representative map $f_k: T \to J$ such that $E_k = E_{f_k}$. Choose $b_k \in E_{f_k}$ and put $J_k = f_k(T) \subseteq J$. Then 
    \[ 
        E_k = E_{f_k} = b_k + \left( P \cap \bigcap_{j \in J_k} N_j \right),
    \]
    and we have the desired normal form.
\end{proof}

\subsection{Quantifier elimination up to translates of pe formulas}

We are now ready to prove \cref{thm:inf-qe}. First we describe the construction of the sets of parameter tuples $I_\alpha$ for a fixed module $M$.

\subsubsection{Construction of the sets $I_\alpha$}

Let $I_0$ consist of all tuples of $0$s of length $< \theta$. Note $|I_0| = \theta \leq \kappa_0$.

For a limit ordinal $\lambda$, put $I_\lambda = \bigcup_{\beta < \lambda} I_\beta$. By induction, for each $\beta < \lambda$, $|I_\beta| \leq \kappa_\beta$. So 
\[
    |I_\lambda| \leq |\lambda| \cdot \sup_{\beta < \lambda} \kappa_\beta = \kappa_\lambda.
\]

For a successor ordinal $\alpha+1$, suppose we have already defined $I_\alpha$ with $|I_\alpha| \leq \kappa_\alpha$. 
For every $\epsilon,\zeta < \theta$, and every nonempty Boolean atom in $\Bool_\infty(\Cos_{\alpha,\epsilon+\zeta}(I_\alpha))$,
\[
    Q = \bigcap_{i \in I} C_i \setminus \bigcup_{j \in J} D_j,
\]
where each $C_i$ is a coset of $P_i$ and $D_j$ is a coset of $N_j$, 
apply \cref{lem:proj-atom} with $\kappa = \kappa_\alpha$ to obtain at most $2^{\kappa_\alpha}$ points $a^{Q}, b_k^Q$ and subsets $J_k \subseteq J$ such that 
\[ 
    \pi(Q) = \left( \pi(a^Q) + \pi \! \left( \bigcap_{i \in I} P_i \right) \right) \setminus \bigcup_k \left[ \pi(b^Q_k) + \pi \left( \bigcap_{i \in I} P_i  \cap \bigcap_{j \in J_k} N_j \right) \right].
\]
Add all these projected points $\pi(a^Q),\pi(b_k^Q)$ to $I_\alpha$ to form $I_{\alpha+1}$. 

Note that if $P_i = \phi_i(M)$ and $N_j = \psi_j(M)$ then 
\[ \pi\! \left( \bigcap_{i \in I} P_i \cap \bigcap_{j \in J_k} N_j \right) \]
is defined by
\[ 
    \exists \ybar \left( \bigwedge_{i \in I} \phi_i(\xbar,\ybar) \wedge \bigwedge_{j \in J_k} \psi_j(\xbar,\ybar) \right),
\]
which belongs to $\Lambda_{\alpha+1,\epsilon}$.

For each $\epsilon,\zeta < \theta$, there are at most $2^{\kappa_\alpha}$ many such Boolean atoms $Q$, because $Q$ is determined by a choice of a subset of $\Cos_{\alpha,\epsilon+\zeta}(I_\alpha)$ and
\[ 
    |\Cos_{\alpha,\epsilon+\zeta}(I_\alpha)| \leq |I_\alpha| \cdot |\Lambda_{\alpha,\epsilon+\zeta}| \leq \kappa_\alpha \cdot \kappa_\alpha = \kappa_\alpha.
\]
So altogether we have put at most 
\[
    \theta^2 \cdot  2^{\kappa_\alpha} \cdot 2^{\kappa_\alpha} = 2^{\kappa_\alpha}
\]
new points into $I_{\alpha + 1}$.

\subsubsection{Proof of the main theorem}

\begin{proof} [Proof of \cref{thm:inf-qe}]
Let $\phi(\xbar) \in L_{\infty,\theta,\alpha}$ with $|\xbar| = \epsilon < \theta$.

Let $\mc{B}_{\alpha,\epsilon} = \Bool_\infty(\Cos_{\alpha,\epsilon}(I_\alpha))$. We want to show that $\phi(M) \in \mc{B}_{\alpha,\epsilon}$. We argue by structural induction on $\phi$.

If $\phi$ is atomic, then $\phi$ is equivalent to a homogeneous linear equation, so $\phi(M) \in \Cos_{0,\epsilon}(I_0) \subseteq \mc{B}_{0,\epsilon}$.

The negation case follows by closure of $\mc{B}_{\alpha,\epsilon}$ under complements.

Suppose $\phi$ is $\bigwedge_i \psi_i(\xbar)$ or $\bigvee_i \psi_i(\xbar)$ with $\sup_i \rank(\psi_i) = \alpha$.
By induction each $\psi_i(M)$ is in the Boolean algebra of its own rank.
Since both $I_\beta$ and $\Lambda_\beta$ are increasing with $\beta$, each $\psi_i(M)$ is in $\mc{B}_{\alpha,\epsilon}$. 
By completeness of $\mc{B}_{\alpha,\epsilon}$, $\phi(M) \in \mc{B}_{\alpha,\epsilon}$.

Finally, suppose $\alpha = \beta+1$ and $\phi(\xbar) = \exists \ybar \psi(\xbar,\ybar)$, where $|\ybar| = \zeta < \theta$ and $\psi$ is rank $\beta$.
Then $\phi(M) = \pi(\psi(M))$.
By induction $\psi(M) \in \mc{B}_{\beta,\epsilon+\zeta}$.
So $\psi(M)$ is a union of Boolean atoms.
By \cref{lem:proj-atom} and the definition of $I_{\beta+1}$, the projection of each such atom is in $\mc{B}_{\beta+1,\epsilon}$.
Projections commute with unions, so $\phi(M) \in \mc{B}_{\beta+1,\epsilon}$.
\end{proof}

\subsection{Are the parameters necessary?}

As we have noted in the Introduction, the need for model-dependent parameters in
\cref{thm:inf-qe} is different from the first-order case, where every parameter-free formula is equivalent in a fixed module to a finite Boolean combination of parameter-free pp formulas. It is natural to wonder whether the parameters can be eliminated from the statement of \cref{thm:inf-qe}. The following example demonstrates that the parameters are indeed necessary, for any normal form using Boolean combinations of formulas that define additive subgroups.

\begin{example}
    There is a module $M$ and a parameter-free $L_{\omega_1,\omega,2}$ formula $\psi(x)$ such that $\psi(M)$ is not in the complete Boolean algebra generated by all additive subgroups of $M$. In particular, $\psi$ is not equivalent in $M$ to any Boolean combination of pe formulas of any rank.
\end{example}

\begin{proof}
    Let $R = \Z[t] / (5t, t^2)$ and let $M = {}_R R$ be the left regular $R$-module. So every element of $M$ can be written uniquely as $n + mt$, where $n \in \Z$ and $m \in \Z/5\Z$.

    Note that $t$ and $2t$ generate the same cyclic subgroup:
    \[ 
        \langle t \rangle = \langle 2t \rangle = \set{ 0, t, 2t, 3t, 4t }.
    \]
    So $t$ and $2t$ appear in exactly the same additive subgroups, and thus no Boolean combination of additive subgroups can distinguish $t$ from $2t$. In particular, no Boolean combination of parameter-free pe formulas of any rank can distinguish $t$ from $2t$.

    On the other hand, there is an $L_{\omega_1,\omega}$ formula $\psi$ that is true of $t$ and false of $2t$. Let $\phi(y)$ be the formula that says $y$ generates $M$ as an $R$-module:
    \[
        \phi(y) = \forall z \bigvee_{r \in R} z = ry, \qquad \phi(M) = \set{\pm 1 + mt \mid m \in \Z/5\Z }.
    \]
    And let $\psi(x)$ be the formula that says that $x = ty$ for some generator $y$:
    \[
        \psi(x) = \exists y (x = ty \wedge \phi(y)), \qquad \psi(M) = \set{t,-t} = \set{t, 4t}.
    \]
    Note that in $\psi(x)$, the $t$ appears not as a parameter but as a unary function symbol in the left $R$-module language. So $\psi$ is a parameter-free $L_{\omega_1,\omega}$ formula, and $\psi$ distinguishes $t$ from $2t$.

    So $\psi$ is not equivalent in $M$ to any Boolean combination of parameter-free pe formulas. On the other hand, note that using $t$ as a parameter, we can write $\psi(x) \equiv \phi_0(x - t) \vee \phi_0(x + t)$, where $\phi_0(x)$ is the pe formula defining the trivial subgroup $x = 0$.
\end{proof}

\section{A stability consequence}

In \cite{Shelah2012}, the pe normal form theorem is used to prove several stability results on modules. Here we record Shelah's proof of one such stability corollary in the notation of the present paper. 

\begin{lemma}
\label{lem:no-long-orders}
    Let $\epsilon,\zeta < \theta$, $\kappa$ infinite, $\mc{C}$ a collection of affine subsets of $M^\epsilon \times M^\zeta$ with $|\mc{C}| \leq \kappa$, and $B \in \Bool_\infty(\mc{C})$.
    There is no pair of sequences $(\abar_i)_{i < (2^\kappa)^+}, (\bbar_i)_{i < (2^\kappa)^+}$ in $M^\epsilon, M^\zeta$ respectively, such that 
    \[
        \forall i < j < (2^\kappa)^+, \quad (\abar_i, \bbar_j) \in B \ \& \ (\abar_j, \bbar_i) \not \in B.
    \]
\end{lemma}

\begin{proof}
    Suppose otherwise. 
    For each $i < j$, since $(\abar_i,\bbar_j)$ and $(\abar_j,\bbar_i)$ disagree on $B$ they disagree on some $C_{ij} \in \mc{C}$. 
    So we may pick a color $(C_{ij}, \sigma_{ij}) \in \mc{C} \times 2$ where $C_{ij}$ is as above and
    \[
    \sigma_{ij} = \begin{cases}
        0, & (\abar_i,\bbar_j) \in C_{ij} \ \& \  (\abar_j,\bbar_i) \not\in C_{ij} \\
        1, & (\abar_i,\bbar_j) \not\in C_{ij} \ \& \ (\abar_j,\bbar_i) \in C_{ij}
    \end{cases}.
    \]
    There are at most $\kappa$ colors, so by the Erd\H{o}s-Rado theorem \cite{ER1956},
    \[ 
        (2^\kappa)^+ \to (4)^2_\kappa,
    \]
    there is a monochromatic set of four indices, say $i_0 < i_1 < i_2 < i_3$ with color $(C,\sigma)$. 
    If $\sigma = 0$ then $(\abar_{i_0}, \bbar_{i_1}), (\abar_{i_0}, \bbar_{i_3}), (\abar_{i_2}, \bbar_{i_3}) \in C$ so
    \[ 
        (\abar_{i_0}, \bbar_{i_1}) + (\abar_{i_2}, \bbar_{i_3}) - (\abar_{i_0}, \bbar_{i_3}) =  (\abar_{i_2}, \bbar_{i_1}) \in C, 
    \]
    which is a contradiction. 
    The case $\sigma = 1$ is symmetric.
\end{proof}

\begin{corollary}[Shelah]
    Let $\phi(\xbar,\ybar)$ be a parameter-free $L_{\infty,\theta,\alpha}$ formula.
    Then $\phi$ does not linearly order any pair of sequences of length $(\kappa_{\alpha+1})^+$ in any left $R$-module.
\end{corollary}

\begin{proof}
    Let $M$ be any left $R$-module.
    Let $\epsilon = |\xbar|$ and $\zeta = |\ybar|$.
    By \cref{thm:inf-qe}, $\phi(M) \in \Bool_\infty(\Cos_{\alpha,\epsilon+\zeta}(I_\alpha))$. 
    And $|\Cos_{\alpha,\epsilon+\zeta}(I_\alpha)| \leq \kappa_\alpha$. 
    So the conclusion follows by applying \cref{lem:no-long-orders} with $B = \phi(M)$ and $\mc{C} = \Cos_{\alpha,\epsilon+\zeta}(I_\alpha)$ and $\kappa = \kappa_\alpha$.
\end{proof}

\end{document}